\documentclass[11pt,a4paper]{amsart}

\usepackage{amsmath}
\usepackage{amssymb}
\usepackage{latexsym}
\usepackage{amsthm}

\usepackage[english]{babel}
\usepackage[utf8]{inputenc}
\usepackage[T1]{fontenc}

\title{Some Graded Identities of The Cayley-Dickson Algebra}

\author{Fernando Henry}
\address{Departamento de Matem\'atica\hfil\break\indent
  Instituto de Matem\'atica e Estat\'\i stica \hfil\break\indent
  Universidade de S\~ao Paulo, Brazil} \email{henry@ime.usp.br}
\urladdr{http://www.ime.usp.br/$\sim$henry}

\theoremstyle{plain}\newtheorem{theo}{Theorem}[section]
\theoremstyle{plain}\newtheorem{lem}[theo]{Lemma}
\theoremstyle{plain}\newtheorem{prop}[theo]{Proposition}
\theoremstyle{plain}\newtheorem{cor}[theo]{Corollary}
\theoremstyle{definition}\newtheorem{defin}[theo]{Definition}
\theoremstyle{remark}\newtheorem{rem}[theo]{Remark}
\theoremstyle{definition}\newtheorem{notation}[theo]{Notation}
\theoremstyle{definition}
\theoremstyle{definition}

\begin{document}

\begin{abstract}
  We work to find a basis of graded identities for the octonion
  algebra. We do so for the $\mathbb{Z}_2^2$ and $\mathbb{Z}_2^3$
  gradings, both of them derived of the Cayley-Dickson process, the
  later grading being possible only when the characteristic of the
  scalars is not two.
\end{abstract}
\maketitle
\tableofcontents

\section{Definitions and Preliminary Results}

\begin{defin}[Graded Algebra]
  An algebra $A$ over the associative, commutative and unitary ring
  $R$ is said \emph{graded by the group $G$}, or simply
  \emph{$G$-graded}, if $A=\underset{a\in G}{\bigoplus}A_a$, as
  $R$-submodules and $A_aA_b\subseteq A_{ab}\ \forall a,b\in G$. We'll
  denote by $a_h$ the projection of $a$ in $A_h$.
\end{defin}

\begin{notation}[Graded Polynomial]
  Let $X$ be a set, $G$ a group and $R$ an associative, commutative
  and unitary ring, we denote by $V[X_G]$ the free groupoid freely
  generated by $X_G:=\{x^a\vert x\in X,\ a\in G\}$ (resp. $V[X_G]^\#$
  for the unitary case) and $R_G\{X\}:=RV[X_G]$ 
  (resp. $R_G\{X\}^\#:=RV[X_G]^\#$) the groupoid ring of $V[X_G]$ by
  $R$ (resp. $V[X_G]^\#$ by $R$). From now on we set $X=\{x_n\vert
  n\in\mathbb{N}\}$ and call $R_G\{X\}$ 
  (resp. $R_G\{X\}^\#$) the $G$-graded non associative polynomial ring
  (resp. the unitary $G$-graded non associative polynomial ring) over
  $R$.
  
  Set $g:V[X_G]\to G$ (resp. $g:V[X_G]^\#\to G$) recursively as
  $g(x^a):=a\ \forall x\in X$ and $g(uv):=g(u)g(v)$ (also let
  $g(1)=e$ the neutral element of $G$ for the $V[X_G]^\#$ case) and
  $R_G\{X\}_a:= \mathrm{lin.span}<g^{-1}(a)>$ (resp. $R_G\{X\}_a^\#:=
  \mathrm{lin.span}<g^{-1}(a)>$). Which makes $R_G\{X\}$
  (resp. $R_G\{X\}^\#$) into the $G$-graded free algebra freely
  generated by $X$ (resp. the unitary $G$-graded free algebra freely
  generated by $X$). From now on we drop the superscript of the
  variables and refer to them by $g$.
  
  Let $x,x_1,\ldots,x_n\in V[X_G]$ and $h,h_1,\ldots,h_n\in G$, set
  $\deg:V[X_G]\to G$ as $\deg u=$ the degree of $u$, $\deg_x:V[X_G]\to
  G$ as $\deg_xu=$ the degree of $u$ with respect to $x$ and
  $\deg_h:V[X_G]\to G$ as $\deg_h u$ as the degree of $u$ with respect
  to all $x$'s such that $g(x)=h$.  Let $f\in R_G\{X\} \setminus\{0\}$
  (resp. $f\in R_G\{X\}^\#\setminus\{0\}$) we shall denote by $\deg f$
  as $\max\{\deg u$| where $u$ is a monomial of $f\} $ and
  $\underline{\deg}f$ as $\min\{\deg u$| where $u$ is a monomial of
  $f\}$, analogue for $\deg_x$, $\underline{\deg}_x$, $\deg_h$ and
  $\underline{\deg}_h$.
  
  We shall call $f$ \emph{homogeneous} if $\deg f=\underline{\deg}f$,
  \emph{homogeneous in $x_1,\ldots,x_n$} (resp. \emph{homogeneous in
    $h_1,\ldots,h_n$}) if $\deg_{x_i}f=\underline{\deg_{x_i}}f$
  (resp. $\deg_{h_i}f=\underline{\deg_{h_i}}f$) for $i=1,
  \ldots,n$. Finally we shall say that $f$ is \emph{multihomogeneous}
  (resp. \emph{multicomponent homogeneous}) if it is homogeneous for
  every $x\in V[X_G]$ (resp. $h\in G$).
\end{notation}

\begin{defin}[Graded Polynomial Identity]
  Let $f(x_1,\ldots,x_n)\in R_G\{X\}$ (resp. $f(x_1\ldots,x_n)\in
  R_G\{X\}^\#$), $f$ is said a \emph{$G$-graded polynomial identity},
  \emph{polynomial identity}, \emph{$G$-P.I.} or simply an
  \emph{P.I.} of the $G$-graded $R$-algebra $A$ if for any $a_i\in
  A_{g(x_i)}\ i=1,\ldots,n\ f(a_1,\ldots,a_n)=0$. The set of all
  $G$-graded polynomial identity of the $G$-graded $R$-algebra $A$
  is called the $G$-graded T-ideal of $A$ and denoted by $T_G(A)$,
  it's easy to see that $T_G(A)$ form an ideal Which is invariant
  under any $G$-graded endomorphism.
  
  $T_G(A)$ is said \emph{homogeneous} (resp. \emph{multihomogeneous},
  \emph{multicomponent homogeneous}) if every homogeneous
  (resp. multihomogeneous, multicomponent homogeneous) of an
  polynomial in $T_G(A)$ lies in $T_G(A)$.
\end{defin}

The next definition and theorem are due to Shirshov on his search for
the answer of the Kurosh problem for alternate P.I. algebras and can
be found in \cite{shir1} and \cite{shir2}. For a long time those
articles were only available in russian and those proof could only be
found in english on \cite{4russos}. However, recently several papers
of Shirshov, including those two, received a translation to english in
\cite{selshir}.

\begin{defin}[r-words]
  Suppose that $X_G$ is ordered. Define recursively $<x_1>:=x_1$,
  $<x_1,\ldots,x_n,x_{n+1}>:=<x_1,\ldots,x_n>\cdot x_{n+1}$ for
  $n\geq1$. We shall call a non associative word of the form $<x_{i_1},
  \ldots,x_{i_n}>$ an \emph{$r_1$-word}. If the $r_1$-word $<x_{i_1},
  \ldots, x_{i_n}>$ is such that $i_1\leq\ldots\leq i_n$ then we shall
  call it an \emph{regular} $r_1$-word. Furthermore, we shall call a
  non associative word of the form $<u_i,\ldots,u_n>$, where each
  $u_i$ is an $r_1$-word (resp. a regular $r_1$-word), an \emph{
    $r_2$-word} (resp. a \emph{regular $r_2$-word}).
\end{defin}

\begin{theo}[Shirshov]
  Let $A$ be an alternative algebra and $v(x_1,\ldots,x_n)$ a non
  associative word. Then for any elements $a_1,\ldots,a_n\in A$ the
  element $v(a_1,\ldots,a_n)$ is representable in the form of a
  linear combination of regular $r_2$-words from $a_1\ldots,a_n$
  with the same length as $v$.
\end{theo}

The following two assertions are well known results in P.I. theory. An
non graded proof of them can be found in \cite{4russos}.

\begin{prop}\label{central}
  Let $A$ be a $G$-graded algebra over an infinite domain $F$,
  torsion free as an $F$-module and $K$ an extension domain of
  $F$. Then $T_G(A)$ is multihomogeneous, furthermore if $A$ is free
  as an $F$-module then $T_G(A)=T_G(A\bigotimes K)$ as algebras over
  $F$.
\end{prop}

\begin{lem}\label{myfirstbaby}
  Let $A$ be a $G$-graded algebra over an infinite domain $F$ and
  torsion free as an $F$-module. Suppose that we have $\mu:G^2\to F$
  and $\nu:G^3\to F$ such that, $xy-\mu(g(x),g(y))yx=0$ and $(xy)z-
  \nu (g(x), g(y), g(z)) x(yz) =0$ are $G$-graded identities of
  $A$. Then $T_G(A)$ is generated by the two above ``scheme''
  identities and possibly some nilpotent identities.
\end{lem}
\begin{proof}
  Let $u$ by a monomial, $J$ the $T$-ideal generated by
  $xy-\mu(g(x),g(y))yx$ and $(xy)z-\nu(g(x),g(y),g(z))x(yz)$. We are
  now  going to show that $u\equiv\lambda w \pmod{J}$, where $w$ is a
  regular $r_1$-word and $\lambda\in F$, for any order we put on the
  variables. Which proves the lemma, by proposition \ref{central}.
  
  We shall do so by  induction on the degree of $u$. Every monomial
  of degree one is a regular $r_1$-word and the identity $xy-\mu
  (g(x),g(y))yx=0$ takes care of the degree two, this proves the
  initial case. Suppose that we have already proved the assertion for
  all words of degree less than $n$ ($n>2$), then it is true for all
  words of degree up to $n$, with effect:
  
  Let $u$, be a monomial of degree $n$ then $u=v_1 s_1$ for some $v_1$,
  $s_1$ monomials of lesser degree. By the induction hypothesis we
  have that $v_1 \equiv \lambda_1 v \pmod{J} $ and $s_1 \equiv
  \lambda_2 s \pmod{J}$, $v$, $s$ regular $r_1$-words. $v= v' x$ and
  $s= s' y$, where $x$ (resp. $y$) is the greatest element of $v$
  (resp. $s$), by definition. If $x> y$ we have that $\lambda_1 v' x
  \cdot \lambda_2 s' y \equiv \lambda_1 \lambda_2 \nu (g(v', g(x),
  g(s'y) ) \mu (g(x), g(s'y) ) (v' \cdot s' y) x \pmod{J}$, if $x \leq
  y$ we have that $\lambda_1 v' x \cdot \lambda_2 s' y \equiv
  \lambda_1 \lambda_2 \mu (g(v's), g(s'y) ) \nu (g(s'),$ $g(y), g(v'x)
  )$ $\mu (g(y), g(v'x) ) (s' \cdot v' x) y \pmod{J}$.
  
  In any case we have that $u \equiv \gamma lz \pmod{J}$, where
  $\gamma \in F$, $l$ is a monomial of degree $n-1$ and $z$ is the
  greatest element of $u$. Finally, by the induction hypothesis, $l
  \equiv \sigma w \pmod{J}$, where $\sigma \in F$ and $w$ is a regular
  $r_1$-word, Which proves the assertion and therefore the lemma.
\end{proof}

\begin{defin}[Composition Algebra]
  A function $n$, from the $F$-vector space $A$ to the field $F$, is
  called a \emph{quadratic form} if $n(\lambda x)=\lambda^2n(x)$ and
  $f(x,y):=n(x+y)-n(x)-n(y)$ is a bilinear form, $\lambda\in F,\
  x,y\in A$.  Furthermore if $A$ is an algebra then $A$ is said a
  \emph{composition algebra} if:
  \begin{itemize}
  \item $n(xy)=n(x)n(y)\ \forall x,y\in A$;
  \item the form $n$ is strictly non degenerate, i.e., f is non
    degenerated;
  \item $A$ is unitary.
  \end{itemize}
\end{defin}

Hurwitz was the first to obtain a classification of finite dimension
composition algebras for the case of the field of complex numbers in
\cite{comp1}, later Dickson gave another proof that carried over to
any algebraically closed field of characteristic not two in
\cite{comp2}, finally in \cite{comp3} Albert obtained a proof for any
field. Further Albert weakened the non degeneracy of $f$ and obtained
a new class of solutions when the field has characteristic two.

The first to study infinite dimensional composition algebras was
Kaplasky in \cite{kap}, and proved that it has to be finite
dimensional, if the non degeneracy of $f$ is weakened then the
composition algebra can also be a purely inseparable quadratic
extension of the field, being of characteristic two and the form
$f(x)=x^2$. Finally Jacobson in \cite{jake} study the automorphisms of
composition algebras and, beside other things, narrowed down the
isomorphisms classes of composition algebras.

For the reminder of this section we'll recall some results of those
articles. The treatment we use is the same one found in
\cite{4russos}.

\begin{notation}
  Denote $\bar{a}:=f(1,a)-a$, $t(a):=a+\bar{a}$ and $n(a):=
  a\bar{a}$.
\end{notation}

\begin{itemize}
\item Every composition algebra is alternative, that is, they satisfy
  the identity $(x,x,y)=(x,y,y)=0$ where $(x,y,z):=(xy)z-x(yz)$ is
  the associator;
\item The map $a\to\bar{a}$ is an involution which leaves the
  elements of $F$ fixed;
\item The elements $t(a)=a+\bar{a}$ and $n(a)=a\bar{a}$ lie in $F$;
\item Every composition algebra satisfy the equality $a^2-t(a)a+
  n(a)=0$.
\end{itemize}

\begin{defin}[Cayley-Dickson process]
  Let $A$ be an unitary $F$-algebra with an involution $a\to\bar{a}$,
  where $a+\bar{a},\ a\bar{a}\in F\ \forall a\in A$ and $\alpha\in
  F\setminus\{0\}$. We shall now construct a new algebra
  $(A,\alpha)$ which involution satisfying the same conditions of
  $A$, therefore we can apply the Cayley-Dickson process on 
  $(A,\alpha)$. Moreover it contains an isomorphic copy of $A$. 
  
  $(A,\alpha):=A\oplus A$ as vector spaces, $(a_1,a_2)(a_3,a_4)
  :=(a_1a_3+\alpha a_4\bar{a}_2,\bar{a}_1a_4+a_3a_2)$ as the
  multiplication and $\overline{(a_1,a_2)}:=(\bar{a}_1,-a_2)$ as the
  involution, clearly $(1,0)$ is the identity element of
  $(A,\alpha)$. We also denote $(1,0)$ and $(0,1)$ simply by $1$ and
  $v$ respectively, so $(a_1,a_2)$ is also denoted by $a_1+va_2$.
  
  If the quadratic form $n(a)=a\bar{a}$ is strictly non degenerate on
  $A$ then $n(x):=x\bar{x}$ in strictly non degenerate on
  $(A,\alpha)$. Moreover if $A$ is a composition algebra, then
  $(A,\alpha)$ is a composition algebra if and only if $A$ is
  associative. Finally if $A$ is $G$-graded (every algebra is graded
  by the trivial group), then $(A,\alpha)$ is
  $(G\times\mathbb{Z}_2)$-graded, as follow: $(A,\alpha)_{(h,0)}:=A_h$
  and $(A,\alpha)_{(h,1)}:=vA_h$.
\end{defin}

We now give four examples of composition algebras:
\begin{enumerate}
\item The field $F$ with $n(x)=x^2$ if char$F\neq2$, otherwise
  $f(x,y)\equiv0$.
\item $\mathbf{K}(\mu):=F\oplus Fv_1$ as vector spaces, $(a+bv_1)
  (c+dv_1):=ac+\mu bd+(ad+bc+bd)v_1$ as multiplication and
  $\overline{a+bv_1}=(a+b)-bv_1$, where $4\mu+1\neq0$. If char$F\neq2$
  then $\mathbf{K}(\mu)=F\oplus vF=(A,\alpha)$ where $v=v_1-2^{-1}$
  and $\alpha=\mu+4^{-1}\neq0$. Conversely, if char$F\neq2$ then
  $(A,\alpha)=F\oplus vF= \mathbf{K}(\mu)$, where $v_1=v+2^{-1}$ and
  $\mu=\alpha-4^{-1}$ also $4\mu+1\neq0$. 
\item $\mathbf{Q}(\mu,\beta):=(\mathbf{K}(\mu),\beta)$ with
  $\beta\neq0$, this is the \emph{algebra of generalized
    quaternions}. It's easy to see that $\mathbf{Q}(\mu,\beta)$ is
  associative but not commutative.
\item $\mathbf{C}(\mu,\beta,\gamma):=(\mathbf{Q}(\mu,\beta),\gamma)$
  with $\gamma\neq0$ is the \emph{Cayley-Dickson algebra} or simply
  the \emph{Octonions} and it is also denoted by $\mathbb{O}$. It's
  easy to see that the octonions are not  associative, therefore we
  cannot continue the Cayley-Dickson process to produce other
  composition algebras.
\end{enumerate}

\begin{lem}\label{Hurwitz lemma}
  Let $B$ be a subalgebra with $1$ of the composition algebra $A$
  and $a,b\in B,\ v\in B^{\perp}$. Then we have the following
  relations: 
  \begin{align*}
    \bar{v}=-v,\qquad& av=v\bar{a}; \tag{A}\\
    a\cdot vb=v\cdot\bar{a}b,\qquad& vb\cdot a=v\cdot ab;\tag{B}\\
    va\cdot vb=v^2\cdot b\bar{a}.& \tag{C}
  \end{align*}
\end{lem}

\begin{theo}[Generalized Hurwitz]
  Let $A$ be a composition algebra. Then $A$ is isomorphic to one of
  the four mentioned composition algebras above.
\end{theo}

\begin{lem}
  For a composition algebra $A$ the following conditions are equivalent:
  \begin{itemize}
  \item $n(x)=0$ for some $0\neq x\in A$;
  \item there are zero divisors in $A$;
  \item $A$ contains an idempotent $e\neq0,1$.
  \end{itemize}
  
  Such a composition algebra is said \emph{split}. 
\end{lem}

\begin{theo}\label{split}
  Any two split composition algebra of the same dimension over a field
  $F$ are isomorphic. Furthermore every composition algebra over an
  algebraically closed field is split.
\end{theo}

\section{Some Identities}
Our goal here is to encounter all the $\mathbb{Z}_2^2$-graded
identities, here the grading is given by the Cayley-Dickson
process. For that we first look at the $\mathbb{Z}_2^3$-graded
identities (obviously for that the field cannot have characteristic
two).  There are two great things about the $\mathbb{Z}_2^3$ grading,
first all the non zero $\mathbb{Z}_2^3$ homogeneous elements are
invertible, second soon we'll know all it's $\mathbb{Z}_2^3$-graded
identities.

We can digest a good part of lemma \ref{Hurwitz lemma} relations into
graded identities. We first note that $g(v)\notin H:=\langle g(a),g(b)
\rangle$ imply that $1\in B=\bigoplus_{h\in H}\mathbb{O}_h$ and
$v\in B^{\perp}$, $B$ is clearly a subalgebra. With that we'll slash
\ref{Hurwitz lemma} hypotheses. There  is still the involution, but it
can be overcame in virtue of (A), as follow:

\begin{prop}\label{z2a3}
  Let $F$ be an infinite field whose characteristic is not two. Then
  $T_{\mathbb{Z}_2^3}\mathbb{O}$ is generated by:
  \begin{align}
    [x_1,x_2]=0, \qquad& \vert\langle g(x_1),g(x_2)\rangle\vert
    \leq2; \label{comutativa} \\
    x_1\circ x_2=0, \qquad& \vert\langle g(x_1),g(x_2)\rangle\vert
    \geq4; \label{anti-comutativa} \\
    (x_1,x_2,x_3)=0, \qquad& \vert\langle g(x_1),g(x_2),g(x_3)\rangle
    \vert\leq4; \label{associativa} \\
    (x_1x_2)x_3+x_1(x_2x_3)=0, \qquad&
    \langle g(x_1), g(x_2), g(x_3) \rangle =
    \mathbb{Z}_2^3. \label{anti-associativa}
  \end{align}
\end{prop}
\begin{proof}
  By lemma \ref{Hurwitz lemma} we have that $\mathbb{O}$ satisfies the
  above identities therefore it's under the conditions of the
  proposition \ref{myfirstbaby} for the $\mathbb{Z}_2^3$
  grading. Furthermore it cannot have any nilpotent identity, since
  every homogeneous element is invertible, which proves the
  proposition.
\end{proof}

\begin{cor}
  Let $D$ be an infinite domain whose characteristic is not two and
  form the ``Cayley-Dickson'' algebra over $D$, $\mathbb{O}$. Then
  $T_{\mathbb{Z}_2^3}\mathbb{O}$ is generated by identities
  \eqref{comutativa}-\eqref{anti-associativa} .
\end{cor}
\begin{proof}
  A direct application of \ref{central}.
\end{proof}

Now we will enter the $\mathbb{Z}_2^2$ realm. The identities bellow
are obtained in the same way that we used to obtain the
$\mathbb{Z}_2^3$ identities.

\begin{align}
  ab\cdot v=v\cdot ba, & \qquad
  g(v)\neq0\neq g(a)=g(b);\label{cocomuta}\\
  (ax\cdot b)v=v(ba\cdot x), & \qquad
  g(v)\neq0=g(x)\neq g(a)=g(b);\label{cocomuta2}\\
  v(ax\cdot b)=(ba\cdot x)v, & \qquad
  g(v)\neq0=g(x)\neq g(a)=g(b);\label{cocomuta3}\\
  x\circ y=0 & \qquad
  \langle g(x),g(y)\rangle=\mathbb{Z}_2^2; \label{antcom}\\
  vb\cdot a=v\cdot ab, & \qquad
  g(v)\notin\langle g(a),g(b)\rangle;\label{entdir}\\
  a\cdot vb=v\cdot ba, & \qquad
  \langle g(v),g(b)\rangle=\mathbb{Z}_2^2,\
  g(a)=0; \label{entesq0?}\\
  va\cdot w+wa\cdot v=-(v\circ w)a, & \qquad
  g(v),g(w)\notin\langle\ g(a)\rangle\neq(0); \label{lokao2}\\
  va\cdot wb+wa\cdot vb=-(v\circ w)ba, & \qquad
  g(v),g(w)\notin\langle\ g(a)(\neq0),g(b)\rangle. \label{idnova}
\end{align}
Beside those we also have:
\begin{align}
  (x,y,z)=0, & \qquad
  \vert\langle g(x),g(y),g(z)\rangle\vert\leq2; \label{ass}\\
  [x,y]=0, & \qquad
  g(x)=g(y)=0; \label{com}\\
  (x,x,y)=(x,y,y)=0; & \qquad
  \label{alt}\\
  v\cdot wb+w\cdot vb=(v\circ w)b, & \qquad
  g(v),g(w)\notin\langle g(b)\rangle.\label{lokao}
\end{align}

Now let $I$ be the $T_{\mathbb{Z}_2^2}$-ideal generated by
\eqref{cocomuta}-\eqref{com}, it's easy to see that \eqref{alt} and
\eqref{lokao} are consequences of \eqref{cocomuta}-\eqref{com}. Our
goal is to prove that $I= T_{\mathbb{Z}_2^2} (\mathbb{O})$. Here on
forward we will simply say that $a$ is equivalent to $b$ or $a \equiv
b$ instead of $a$ is equivalent to $b$ modulo $I$ or $a \equiv b
\pmod{I}$.

The basic idea of the proof is to assume, by contradiction, that
$I \neq T_{\mathbb{Z}_2^2}$, then there is a $f \in T_{\mathbb{Z}_2^2}
(\mathbb{O}) \setminus I$ of minimal degree. Following with an
appropriate substitutions slice up $f$ in several identities $f_i$,
each being a consequence of some identity $g_i$ of lesser degree,
therefore in $I$ contradicting that $f \notin I$.

The first step is to reduce every monomial to a normal form, in virtue
of Shirshov's Theorem, it's enough to consider $u$ a regular
$r_2$-word ($x_n^\nu< x_m^\mu$ if $\nu< \mu$ or $\nu= \mu$, $n<m$, for
now we will only say that $0$ is the greatest element of
$\mathbb{Z}_2^2$). Strictly speaking we don't need to use Shirshov's
Theorem, however it will save us half the work, so we'll gladly use it.

It is worth noting that all the identities that generate $I$ are
multilinear therefore $I$ is multihomogeneous and the equivalence
preserves multidegree.

\section{The Zero Component Variables}

\begin{lem}\label{lem0}
  Let $u$ be a regular $r_2$-word and $x$ the greatest element that
  $u$ depends on, suppose that $g(x)=0$. Then we have the following
  possibilities:
  \begin{itemize}
  \item $u\equiv\pm yx$;
  \item $u\equiv\pm xy$, $g(u)\neq0$;
  \item $u\equiv\pm yx\cdot z$, $g(y)=g(z)\neq0$;
  \end{itemize}
  where $y,z$ are monomials.
\end{lem}
\begin{proof}
  We have that $u=(\ldots ((u_1u_2)u_3)\dots u_{n-1})u_n$ where each
  $u_i$ is a regular $r_1$-word. We'll prove the lemma by induction
  on $n$. The initial case is exactly $yx$. If $x$ appears on $u_n$,
  $n\neq1$ then we need only to agglutinate what's to the left of
  $x$, that is, we have $z\cdot yx$ and want to obtain $wx$ or $xw$
  (just to make things crystal clear, $z=(\ldots ((u_1u_2)u_3)\dots
  u_{n-2})u_{n-1}$ and $u_n=yx$). This part of the proof (as many
  more to come) is divided into cases. Each case is one, or more,
  possibilities of the homogeneous component of each variable.
  \begin{itemize}
  \item $\langle g(z),g(y)\rangle\neq\mathbb{Z}_2^2$\\
    $z\cdot yx\equiv_{\eqref{ass}} zy\cdot x$,
  \item $\langle g(z),g(y)\rangle=\mathbb{Z}_2^2$\\
    $z\cdot yx\equiv_{\eqref{entesq0?}} x\cdot zy$.
  \end{itemize}
  
  Now if $x$ doesn't appear in $u_n$ we have, by the induction
  hypothesis, that $u\equiv\pm yx\cdot z$, $u\equiv\pm xy\cdot z$ or
  $(tx\cdot w)z$, where $g(t)=g(w)\neq0$. We are going to divide
  those three cases into sub cases. Each sub case is one, or more,
  possibilities of the homogeneous component of each variable. We
  begin which the first two cases:
  \begin{itemize}
  \item $\langle g(z),g(y)\rangle=\mathbb{Z}_2^2$\\
    1.\ $xy\cdot z\equiv_{\eqref{antcom}}-z\cdot xy
    \equiv_{\eqref{entdir}}-zy\cdot x$,\quad
    2.\ $yx\cdot z\equiv_{\eqref{antcom}}-z\cdot yx
    \equiv_{\eqref{entesq0?}}-x\cdot zy$;
  \item $g(y)=0$\\
    1.\ $xy\cdot z\equiv_{\eqref{ass}}x\cdot yz$,\quad
    2.\ $yx\cdot z\equiv_{\eqref{com}}xy\cdot z
    \equiv_{\eqref{ass}}x\cdot yz$;
  \item $g(z)=0$\\
    1.\ $xy\cdot z\equiv_{\eqref{ass}}x\cdot yz$,\quad
    2.\ $yx\cdot z\equiv_{\eqref{ass}}y\cdot xz
    \equiv_{\eqref{com}}y\cdot zx\equiv_{\eqref{ass}}yz\cdot x$;
  \item $g(y)=g(z)\neq0$\\
    1.\ $xy\cdot z\equiv_{\eqref{ass}}x\cdot yz$,\quad
    2.\ $yx\cdot z$. 
  \end{itemize}
  
  We now proceed to the last case:
  \begin{itemize}
  \item $g(z)=0$\\
    $(tx\cdot w)z\equiv_{\eqref{ass}}tx\cdot wz$;
  \item $g(z)=g(t)$\\
    $(tx\cdot w)z\equiv_{\eqref{ass}}t(x\cdot wz)
    \equiv_{\eqref{com}}t(wz\cdot x)\equiv_{\eqref{ass}}(t\cdot wz)x$;
  \item $\langle g(t),g(z)\rangle=\mathbb{Z}_2^2$\\
    $(tx\cdot w)z\equiv_{\eqref{cocomuta2}}z(wt\cdot x)\equiv_{
      \eqref{ass}}(z\cdot wt)x$.
  \end{itemize}
\end{proof}

\begin{cor}\label{forma0}
  Let $f$ be a multihomogeneous polynomial, $x$ the greatest
  element that $f$ depends on and $n=\deg_xf$, where $g(x)=0$. Then we
  have one of the following:
  \begin{itemize}
  \item $f \equiv \Sigma_{i=0}^n x^i y_i x^{n-i}$ if $g(f) \neq0 $
    therefore $g(y_i) =g(f) \neq 0$;
  \item $f \equiv \Sigma_{i=0}^n \Sigma_j y_{i,j} x^i \cdot z_{i,j}
    x^{n-i}$ if $g(f) =0$, where $g(y_{i,j}) =g(z_{i,j}) \neq 0$.
  \end{itemize}
\end{cor}
\begin{proof}
  We shall prove the corollary by induction on $n$, by Shirshov's
  theorem we may assume that all of $f$ monomials are regular
  $r_2$-words, the initial case is just the lemma \ref{lem0}, which is
  already proved. Suppose that the assertion is valid for polynomials
  of degree $n$ then it is valid for polynomials of degree $n+1$, with
  effect, let $\deg_xf =n +1$.
  
  Suppose that $g(f) \neq 0$ then $f \equiv px + xh$ by lemma
  \ref{lem0} and by the induction hypothesis we have that $px \equiv (
  \Sigma_{i=0}^n x^i p_i x^{n-i}) x$ and $xh \equiv x (\Sigma_{i=0}^n
  x^i h_i x^{n-i})$, which proves this case.
  
  Suppose that $g(f) =0$ then $f \equiv px + \Sigma_j y_j x \cdot z_j$
  by lemma \ref{lem0} and by the induction hypothesis we have that $p
  \equiv \Sigma_{i=0}^n \Sigma_j u_{i,j} x^i \cdot w_{i,j} x^{n-i}$,
  $y_j \equiv \Sigma_{i=0}^{n_j} x^i p_{i,j} x^{n_j-i}$ and $z_j
  \equiv \Sigma_{i=0}^{m_j} x^i h_{i,j} x^{m_j-i}$, s.t., $m_j +n_j
  =n$. Which proves this case and with that the corollary.
\end{proof}

\begin{prop}\label{prop0}
  Let $f$, $x$ and $n$ be as in \ref{forma0}, suppose that $f \in T_{
    \mathbb{Z}_2^2} (\mathbb{O})$ and $F$ is an infinite field. Then
  $y_i \in T_{\mathbb{Z}_2^2} (\mathbb{O})$ for $i =1, \ldots, n$ in
  the first case of \ref{forma0} or $\Sigma_j y_{i,j} x \cdot z_{i,j}
  \in T_{\mathbb{Z}_2^2} (\mathbb{O})$ for $i= 1, \ldots, n$ in the
  second case of \ref{forma0}.
\end{prop}
\begin{proof}
  If $g(f) \neq 0$ (resp. $g(f) =0$) we have, by \ref{forma0}, that $f
  \equiv \Sigma_{i=0}^n x^i y_i x^{n-i}$ where $g(y_i) =g(f) \neq0$
  (resp. $f \equiv \Sigma_{i=0}^n \Sigma_j y_{i,j} x^i \cdot z_{i,j}
  x^{n-i}$, where $g(y_{i,j}) =g(z_{i,j}) \neq0$). By \ref{central} we
  can assume that $F$ is algebraically closed, therefore $\mathbb{O}_0
  \cong F \oplus F$ where $\overline{(a,b)} =(b,a)$.
  
  Under any evaluation of $f$ we have that $f =y_i \bar{x}^i x^{n-i}$
  (resp. $f= p_i \bar{x}^i x^{n-i}$ where $p_i =\Sigma_j y_{i,j}
  z_{i,j}$). Let $x=(x_1, x_2)$ and $y_i= v(y_i', y_i'')$ where $v$ is
  given by the Cayley-Dickson process (resp. $p_i= (p_i'
  ,p_i'')$). Then $f= v(\Sigma_i^n y_i' x_1^{n-i} x_2^i,$ $\Sigma_i^n
  y_i' x_1^i x_2^{n-i})$ (resp. $f= (\Sigma_i^n p_i' x_1^{n-i} x_2^i,
  \Sigma_i^n p_i' x_1^i x_2^{n-i})$). Let $x_1, x_2$ be algebraically
  independent variables over $F$ then $y_i' =y_i'' =0$ (resp. $p_i'
  =p_i'' =0$, $0 =p_i =xp_i =x \Sigma_j y_{i,j} z_{i,j} \equiv
  \Sigma_j y_{i,j} x \cdot z_{i,j}$) under any substitution in $F$.
\end{proof}

\begin{prop}\label{comp0sotem1}
  Let $f \in T_{\mathbb{Z}_2^2} (\mathbb{O}) \setminus I$
  multihomogeneous of minimal degree. Then $\deg_0f \leq 1$.
\end{prop}
\begin{proof}
  It's enough to consider the case $f= \Sigma_i y_i (x_1 \cdots x_n)
  \cdot z_i$ where $g(x_j) =0 \neq g(y_i) =g(z_i)$, $j=1, \ldots, n$
  and $\deg_0 f= n$, by induction and \ref{prop0}. Substituting $x_2,
  \ldots, x_n$ for $1$ we see that $\Sigma_i y_i x_1 \cdot z_i$ is an
  identity.
  
  Suppose, by contradiction, that $n > 1$ therefore $\Sigma_i y_i x_1
  \cdot z_i \in I$, by the minimality of $f$'s degree. If we let $x_1$
  go to $x_1 \cdots x_n$ we see that $f$ is a consequence of $\Sigma_i
  y_i x_1 \cdot z_i$, which is a contradiction.
\end{proof}

\section{The Strictly Non Zero Component Variables}
\begin{defin}
  Let $U$ be the polynomial sub-algebra generated by all variables
  that aren't from the zero component and $ ^*:U\to U$ linear defined
  on monomials by induction on the degree as follow: $u^*:=-u$ if
  $\deg u=0$ and $(vw)^*:=w^*v^*$.
\end{defin}

\begin{lem}\label{invol}
  Let $f\in U$ and $x$ a non zero component variable. Then we have the
  following:
  \begin{enumerate} 
  \item $^*$ is an involution of $U$;
  \item if $f_0=0$ then $f^*\equiv-f$;
  \item if $f=v\cdot w$, $g(w)=g(v)\neq0$ then $f^*\equiv w\cdot v$;
  \item if $f_{g(x)}=0$ then $fx\equiv xf^*$;
  \item $f^*$ goes to $\bar{f}$ under any evaluation;
  \end{enumerate}
\end{lem}
\begin{proof}
  By linearity it's enough to prove the lemma only for the case where
  $f$ is a monomial.
  
  (1) We'll start proving by induction that $^*$ has
  order 2. If $\deg f=1$ then $f^{*^*}=(-f)^*=-(-f)=u$ if $\deg
  f\neq1$ then $f=v\cdot w$ and $f^{*^*}=(v\cdot w)^{*^*}=(w^*\cdot
  v^*)^*=v^{*^*}\cdot w^{*^*}=v\cdot w=f$. Clearly $^*$ is a
  anti-homomorphism.
  
  Both (2) and (4) are valid when $\deg f=1$. Suppose they are valid
  for all monomials of degree less than $n$, then they are valid for
  monomials of degree $n$, with effect:
  
  (2) $f=v\cdot w$, we have the following cases:
  \begin{itemize}
  \item $\langle g(w),g(v)\rangle=\mathbb{Z}_2^2$\\
    $f^*=w^*\cdot v^*=w^*\cdot v^*\equiv(-w)\cdot(-v)$ by the
    induction hypothesis for (2), $(-w)\cdot(-v)\equiv_{\eqref{antcom}}
    -v\cdot w=-f=f^*$ by the induction hypothesis for (4);
  \item $g(w)=0\neq g(v)$\\
    $f^*=w^*\cdot v^*\equiv w^*\cdot (-v)$ by the induction hypothesis
    for (2), $w^*\cdot (-v)\equiv -v\cdot w=-f$ by the induction
    hypothesis for (4);
  \item $g(v)=0\neq g(w)$\\
    $f^*=w^*\cdot v^*\equiv (-w)\cdot v^*$ by the induction hypothesis
    for (2), $(-w)\cdot v^*\equiv -v\cdot w=-f$ by the induction
    hypothesis for (4);
  \end{itemize}
  
  (4) $f=v\cdot w$, we have the following cases:
  \begin{itemize}
  \item $\langle g(f),g(x)\rangle=\mathbb{Z}_2^2$\\
    $fx\equiv_{\eqref{antcom}}-xf\equiv xf^*$ by the induction
    hypothesis for (2);
  \item $g(f)=0$, $g(v)=g(w)\neq0$\\
    $fx=(v\cdot w)x\equiv_{\eqref{cocomuta}}x(wv)=x((-w)(-v))=xf^*$
  \end{itemize}
  
  (3) Clear after we proved (2).
  
  (5) Trivial.
\end{proof}

\begin{lem}\label{lemnon0}
  Let $u$ be a regular $r_2$-word and $x$ the greatest element that
  $u$ depends on, suppose that $g(x)\neq0$. Then we have one of the
  following:
  \begin{itemize}
  \item $u\equiv\pm yx$,
  \item $u\equiv\pm xy$ if $g(x)=g(y)$,
  \item $u\equiv\pm z\cdot yx$ if $g(x)=g(y)$ and $\langle
    g(x),g(z)\rangle=\mathbb{Z}_2^2$ or
  \item $u\equiv\pm z\cdot xy$ if $g(x)=g(y)=g(z)$;
  \end{itemize}
  where $y,z$ are monomials.
\end{lem}
\begin{proof}
  We have that $u=(\ldots ((u_1u_2)u_3)\dots u_{n-1})u_n$ where each
  $u_i$ is a regular $r_1$-word. We'll prove the lemma by induction on
  $n$. The initial case is exactly $yx$. If $x$ appears on $u_n$ and
  $n\neq1$ then we need only to agglutinate what's to the left of $x$,
  that is, we have $z\cdot yx$, where $z=(\ldots ((u_1 u_2) u_3) \dots
  ) u_{n-1}$, $yx=u_n$ and want to obtain $wx$ or $t\cdot sx$  
  with $g(s)=g(x)$ and $\langle g(x),g(t)\rangle=\mathbb{Z}_2^2$:
  \begin{itemize}
  \item $\langle g(x),g(y),g(z)\rangle\neq\mathbb{Z}_2^2$\\
    $z\cdot yx\equiv_{\eqref{ass}}zy\cdot x$;
  \item $\langle g(x),g(y)\rangle=\mathbb{Z}_2^2,\ g(z)=0$\\
    $z\cdot yx\equiv_{\eqref{antcom}}-z\cdot xy\equiv_{
      \eqref{entesq0?}}-x\cdot yz\equiv_{\eqref{antcom}}yz\cdot x$;
  \item $\langle g(x),g(y)\rangle=\mathbb{Z}_2^2,\ g(y)=g(z)$\\
    $z\cdot yx\equiv_{\eqref{antcom}}xy\cdot z\equiv_{
      \eqref{entdir}}x\cdot zy\equiv_{\eqref{cocomuta}}yz\cdot x$;
  \item $\langle g(x),g(y)\rangle=\mathbb{Z}_2^2,\ g(x)=g(z)$\\
    $z\cdot yx\equiv_{\eqref{antcom}}-yx\cdot z\equiv_{
      \eqref{entdir}}-y\cdot zx$;
  \item $\langle g(x),g(z)\rangle=\mathbb{Z}_2^2,\ g(y)=0$\\
    $z\cdot yx\equiv_{(\ref{invol})}z\cdot xy^*\equiv_{
      \eqref{antcom}}-xy^*\cdot z\equiv_{\eqref{entdir}}-x
    \cdot zy^*\equiv_{\eqref{antcom}}zy^*\cdot x$;
  \item $\langle g(x),g(z)\rangle=\mathbb{Z}_2^2,\ g(x)=g(y)$\\
    $z\cdot yx$, nothing to see here, move along;
  \item $\langle g(x),g(y),g(z)\rangle=\mathbb{Z}_2^2,\ 
    g(x)+g(y)+g(z)=0$\\
    $z\cdot yx\equiv_{\eqref{antcom}}-z\cdot xy\equiv_{
      \eqref{lokao},\eqref{antcom}}x\cdot zy$.
  \end{itemize}
  
  If x doesn't appear on $u_n$ then we have that $u\equiv\pm yx\cdot
  z$, $u\equiv\pm xy\cdot z$, $u\equiv\pm(t\cdot sx)w$ or $(z\cdot
  xy)w$  with $g(s)=g(x)$ and $\langle g(x),g(t)\rangle=
  \mathbb{Z}_2^2$ in the third case or $g(x)=g(y)=g(z)$ in the last
  case, by the induction hypothesis, and want to obtain $\pm wx$, $\pm
  xw$ or $\pm t\cdot sx$, with the same component restriction. Let's
  begin with the case $\pm(t\cdot sx)w$.
  \begin{itemize}
  \item $g(w)=0$\\
    $(t\cdot sx)w\equiv_{\eqref{ass}}t(s\cdot xw)\equiv_{
      (\ref{invol})}t(s\cdot w^*x)\equiv_{\eqref{ass}}
    t(sw^*\cdot x)$;
  \item $g(t)=g(w)$\\
    $(t\cdot sx)w\equiv_{(\ref{invol})}(xs\cdot t)w\equiv_{
      \eqref{ass}}x(s\cdot tw)$;
  \item $g(x)=g(w)$\\
    $(t\cdot sx)w\equiv_{\eqref{entdir}}t(w\cdot sx)\equiv_{
      \eqref{ass}}t(ws\cdot x)\equiv_{(\ref{invol})}t(x\cdot sw)
    \equiv_{\eqref{entdir}}(t\cdot sw)x$;
  \item $g(w)=g(x)+g(t)$\\
    $(t\cdot sx)w\equiv_{\eqref{entdir}}(tx\cdot s)w\equiv_{
      \eqref{antcom}}(s\cdot xt)w\equiv_{\eqref{entdir}}s(w\cdot xt)
    \equiv_{ \eqref{lokao},\eqref{antcom}}-s(x\cdot wt)
    \equiv_{\eqref{ass}}-sx\cdot wt\equiv_{(\ref{invol})}
    -wt\cdot xs$.
  \end{itemize}
  
  Moving to $\pm yx\cdot z$ and $\pm xy\cdot z$:
  \begin{itemize}
  \item $g(y)=0,\ \langle g(x),g(z)\rangle\neq\mathbb{Z}_2^2$\\
    $yx\cdot z\equiv_{(\ref{invol})}xy^*\cdot z\equiv_{
      \eqref{ass}}x\cdot y^*z$;
  \item $g(y)=0,\ \langle g(x),g(z)\rangle=\mathbb{Z}_2^2$\\
    $yx\cdot z\equiv_{(\ref{invol})}xy^*\cdot z\equiv_{
      \eqref{entdir}}x\cdot zy^*\equiv_{\eqref{antcom}}-zy^*\cdot x$;
  \item $\langle g(x),g(y)\rangle=\mathbb{Z}_2^2,\ g(z)=0$\\
    $yx\cdot z\equiv_{\eqref{antcom}}-xy\cdot z\equiv_{
      \eqref{entdir}}-x\cdot zy\equiv_{\eqref{antcom}}zy\cdot x$;
  \item $\langle g(x),g(y)\rangle=\mathbb{Z}_2^2,\ g(z)=g(x)$\\
    $yx\cdot z\equiv_{\eqref{entdir}}y\cdot zx$;
  \item $\langle g(x),g(y)\rangle=\mathbb{Z}_2^2,\ g(z)=g(y)$\\
    $yx\cdot z\equiv_{\eqref{antcom}}-xy\cdot z\equiv_{
      \eqref{entdir}}-x\cdot zy\equiv_{(\ref{invol})}-yz\cdot x$;
  \item $\langle g(x),g(y)\rangle=\mathbb{Z}_2^2,\ g(z)=g(x)+g(y)$\\
    $yx\cdot z\equiv_{\eqref{antcom}}-xy\cdot z\equiv_{\eqref{lokao2}
      \eqref{antcom}}zy\cdot x$;
  \item $g(x)=g(y),\ g(z)=0$\\
    1.\ $yx\cdot z\equiv_{\eqref{com}}z\cdot yx
    \equiv_{\eqref{ass}}zy\cdot x$,\quad
    2.\ $xy\cdot z\equiv_{\eqref{ass}}x\cdot yz$;
  \item $g(x)=g(y),\ \langle g(x),g(z)\rangle=\mathbb{Z}_2^2$\\
    1.\ $yx\cdot z\equiv_{(\ref{invol})}z\cdot xy
    \equiv_{\eqref{entdir}} zy\cdot x$,\quad
    2.\ $xy\cdot z\equiv_{(\ref{invol})}z\cdot yx$;
  \item $g(x)=g(y)=g(z)$\\
    1.\ $yx\cdot z\equiv_{\eqref{ass}}y\cdot xz$,\quad
    2.\ $xy\cdot z\equiv_{\eqref{ass}}x\cdot yz\equiv_{(\ref{invol}}
    zy \cdot x$.
  \end{itemize}
  
  We now proceed to the last case:
  \begin{itemize}
  \item $g(w)=0$\\
    $(z\cdot xy)w\equiv_{\eqref{ass}}z(x\cdot yw)$;
  \item $g(x)=g(w)$\\
    $(z\cdot xy)w\equiv_{\eqref{ass}}z(x\cdot yw)\equiv_{
      (\ref{invol})}z(wy\cdot x)\equiv_{\eqref{ass}}(z\cdot
    wy)x$;
  \item $\langle g(x),g(w)\rangle=\mathbb{Z}_2^2$\\
    $(z\cdot xy)w\equiv_{\eqref{antcom}}-w(z\cdot xy)\equiv_{
      \eqref{entesq0?}}-xy\cdot wz\equiv_{(\ref{invol})}
    -wz\cdot yx$.
  \end{itemize}
\end{proof}

\begin{cor}\label{formanon0}
  Let $f$ be a multihomogeneous polynomial, $x$ the greatest element
  that $f$ depends on and $n=\deg_xf$, where $g(x)\neq0$. Then we have
  one of the following:
  \begin{itemize}
  \item $f\equiv px^n+ \Sigma_{i=1}^n\Sigma_jp_{i,i+1}^jx\cdots
    p_{i,1}^jx^{n-i} + \Sigma_{i'=1}^n\Sigma_jxh_{i',i'}^jx\cdots
    h_{i',1}^j x^{n'-i}$, if $g(f)\in\langle g(x)\rangle$, where $g(
    p_{i,l}^j)=g(h_{i',l}^j)=g(x)$, $\forall i,i',j,l$, $(n+i+1)g(x)
    =g(f)$ and $(n+i')g(x)=g(f)$;
  \item $f\equiv zx^n+ \Sigma_{i=1}^n\Sigma_jz_j\cdot p_{i,i+1}^jx\cdots
    p_{i,1}^jx^{n-i} + \Sigma_{i'=1}^n\Sigma_jz_j\cdot xh_{i',i'}^jx\cdots
    h_{i',1}^j x^{n'-i}$, if $\langle g(f),g(x)\rangle=\mathbb{Z}_2^2$,
    where $g(p_{i,l}^j)=g(h_{i',l}^j)=g(x)$, $\langle g(z_j),g(x)
    \rangle=\mathbb{Z}_2^2$, $\forall i,i',j,l$, and $n-i\equiv n-i'
    \equiv 1\pmod{2}$
  \end{itemize}
\end{cor}
\begin{proof}
  We shall prove the corollary by induction on $n$, by Shirshov's
  theorem we may assume that all of $f$'s monomials are regular
  $r_2$-words, the initial case is just the lemma \ref{lemnon0}, which
  is already proved. Suppose that the assertion is valid for
  polynomials of degree $n$ then it is valid for polynomials of degree
  $n+1$, with effect, let $\deg_xf=n+ 1$.
  
  If $g(f)\in\langle g(x)\rangle$ then we have $f\equiv qx+xh+
  \Sigma_{j'} z_{j'}xy_{j'}$, where $g(p),g(h),$ $g(z_{j'}),g(y_{j'})
  \in \langle g(x)\rangle$, by \ref{lemnon0}. We have that $q \equiv
  p^qx^n+ \Sigma_{i=1}^n \Sigma_j p_{i,i+1}^{j,q} x\cdots
  p_{i,1}^{j,q}$ $x^{n-i} + \Sigma_{i'=1}^n \Sigma_j x h_{i',i'}^{j,q} x
  \cdots h_{i',1}^{j,q} x^{n'-i}$, by the induction hypothesis, so $qx
  \equiv  p^qx^{n+1} + \Sigma_{i=1}^n\Sigma_j p_{i,i+1}^{j,q} x \cdots
  p_{i,1}^{j,q}x^{n+1-i} + \Sigma_{i'=1}^n \Sigma_j xh_{i',i'}^{j,q} x
  \cdots h_{i',1}^{j,q} x^{n'+1-i}$. Analogously, $xh \equiv  xp^hx^n
  + \Sigma_{i=1}^n \Sigma_j xp_{i,i+1}^{j,h} x \cdots
  p_{i,1}^{j,h}x^{n-i} + \Sigma_{i'=1}^n \Sigma_j h_{i',i'}^{j,h} x
  \cdots h_{i',1}^{j,h}$ $x^{n'-i+2}$. The last summand is analogous.
  
  Clearly $f\equiv px^n+ \Sigma_{i=1}^n\Sigma_j
  p_{i,i+1}^jx\cdots p_{i,1}^jx^{n-i} +\Sigma_{i'=1}^n\Sigma_jx
  h_{i',i'}^j x\cdots h_{i',1}^j x^{n'-i}$; $g(p_{i,l}^j),
  g(h_{i',l}^j)\in\langle g(x)\rangle$. If $g(p_{i,l}^j)=0$ then
  $p_{i,i+1}^jx \cdots p_{i,1}^j x^{n-i} \equiv p_{i,i+1}^j x \cdots
  x$ $ p_{i,l-1}^j(p_{i,l}^j)^*p_{i,l+1}^jx\ldots p_{i,1}^jx^{n-i+2}$
  if $l \neq i+1$ and $p_{i,i+1}^jx\cdots p_{i,1}^jx^{n-i}\equiv x
  p_{i,i+1}^j$ $p_{i,i}^jx\cdots p_{i,1}^jx^{n-i}$ otherwise. If
  $g(h_{i',l}^j)=0$ then $x h_{i',i'}^j \cdots x h_{i',1}^jx^{n-i'}
  \equiv x$ $h_{i',i'}^j\cdots x h_{i',l-1}^j(h_{i',l}^j)^*
  h_{i',l+1}^j x \ldots h_{i',1}^jx^{n-i'+2}$. So we can assume that
  $g(p_{i,l}^j)=g(h_{i',l}^j)=g(x)$, $\forall i,i',j,l$. Calculating
  $g$ on both sides of the equivalence we obtain that $(n+i+1)g(x)
  =g(f)$ and $(n+i')g(x)=g(f)$.
  
  The case $\langle g(f),g(x)\rangle=\mathbb{Z}_2^2$ is analogous.
\end{proof}

\begin{defin}
  We shall define $v_{(1,0)}$ as $(0,1)\in\mathbf{Q}(\mu,\beta)$,
  $v_{(0,1)}$ as $(0,1)\in \mathbf{C}(\mu,\beta,\gamma)$ given by the
  Cayley-Dickson process, $v_{(1,1)}$ as $v_{(1,0)}v_{(0,1)}$ and
  $v_{(0,0)}$ as $1$. If $a\in\mathbb{O}_h$, $a=v_ha'$, $a'\in
  \mathbb{O}_0$ we define $\tilde{a}:=v_h\bar{a}$ and extend $\tilde{
  }$ to $\mathbb{O}$ by linearity. Finally we define recursively
  $a^{[n]}$ as $a^{[0]}=1$ and $a^{[n+1]}=a^{[n]}\tilde{a}$ if
  $n\equiv 0 \pmod{2}$ or $a^{[n+1]}=a^{[n]}a$ if $n\equiv 1 \pmod{2}$.
\end{defin}

\begin{prop}\label{propnon0}
  Let $f$, $x$ and $n$ be as in \ref{formanon0}, suppose that $f\in
  T_{\mathbb{Z}_2^2}(\mathbb{O})$ and $F$ is an infinite field. Then
  $p$, $\Sigma_jp_{i,i+1}^jx\cdots xp_{i,1}^j$, $\Sigma_j
  h_{i',i'}^j x\cdots  xh_{i',1}^j\in T_{\mathbb{Z}_2^2}(\mathbb{O})$
  for $i,i'=1,\ldots,n$ in the first case of \ref{formanon0} or
  $z$, $\Sigma_jz_j\cdot p_{i,i+1}^jx\cdots p_{i,1}^jx$, $\Sigma_jz_j
  \cdot h_{i',i'}^jx\cdots h_{i',1}^jx\in T_{\mathbb{Z}_2^2}(\mathbb{O})$
  for $i,i'=1,\ldots,n$ in the second case of \ref{formanon0}.
\end{prop}
\begin{proof}
  For the first case we have that under any evaluation on $\mathbb{O}$
  $f=px^n+\Sigma_i^nx^{[i]}a_ix^{n-i}+\Sigma_{i'}^n \tilde{x}^{[i']}
  b_{i'}x^{n-i'}$ where $a_i=\Sigma_jP_{i,i+1}^j\cdots P_{i,1}^j$,
  $b_{i'}=\Sigma_jH_{i',i}^j\cdots H_{i',1}^j$, $P_{i,l}^j= p_{i,l}^j$
  if $l$ is odd or $P_{i,l}^j=\tilde{p}_{i,l}^j$ if $l$ if even and
  $H_{i',l}^j=h_{i',l}^j$ if $l$ is odd or $H_{i',l}^j=\tilde{
    h}_{i',l}^j$ if $l$ is even. Now we have four sub-cases:
  \begin{itemize}
  \item $n$ is odd $g(f)=g(x)$, therefore $i$ is odd and $i'$ is even,
    then $f=x^np^*+\Sigma x^{[i]}x^{n-i}a_i+\Sigma\tilde{x}^{[i']}x^{n-i}
    b_{i'}^*$;
  \item $n$ is even $g(f)=0$, therefore $i$ is odd and $i'$ is even,
    then $f=x^np+\Sigma x^{[i]}x^{n-i}a_i+\Sigma\tilde{x}^{[i']}x^{n-i}
    b_{i'}^*$;
  \item $n$ is odd $g(f)=0$, therefore $i$ is even and $i'$ is odd,
    then $f=\tilde{x}^n\tilde{p}+\Sigma x^{[i]}\tilde{x}^{n-i}
    \tilde{a}_i+\Sigma\tilde{x}^{[i']}x^{n-i}b_{i'}$;
  \item $n$ is even $g(f)=g(x)$, therefore $i$ is even and $i'$ is
    odd, then $f=x^np+\Sigma x^{[i]}\tilde{x}^{n-i}\tilde{a}_i
    +\Sigma\tilde{x}^{[i']}x^{n-i}b_{i'}$.
  \end{itemize}
  In any case, we can use generic elements as was done in \ref{prop0}
  that results in $p=a_i=b_{i'}=0$. Therefore $0=x^{[i]}a_i= \Sigma_j
  p_{i,i+1}^jx \cdots xp_{i,1}^j$ and $0=\tilde{x}^{[i']}b_{i'}=\Sigma_j
  h_{i',i'}^j x\cdots xh_{i',1}^j$.
  
  The second case of \ref{formanon0} is analogous.
\end{proof}

\begin{prop}\label{totalnon0}
  Let $f$ be a multihomogeneous polynomial, $h\in\mathbb{Z}_2^2
  \setminus(0)$ the greatest component that $f$ depends on, $x_i$,
  $i=1,\ldots,n$ the variables from the $h$ component that $f$ depends
  on and $m_i=\deg_{x_i}f$, $m=\Sigma_{i=1}^nm_i=\deg_hf$. Suppose
  that $f\in T_{\mathbb{Z}_2^2}(\mathbb{O})\setminus I$ of minimal
  degree. Then $f$ is of one of the following forms:
  \begin{itemize}
  \item $f\equiv \Sigma_jP_1^j\cdots P_n^j p^j$;
  \item $f\equiv \Sigma_jz^j\cdot P_1^j\cdots P_n^j$;
  \end{itemize}
  where $P_i^j= p_{i,1}^jx_i\cdots p_{i,m_i}^jx_i$, $g(p_{i,l}^j)=
  g(p^j) =g(x)$, $\langle g(z^j),g(x) \rangle=\mathbb{Z}_2^2$.
\end{prop}
\begin{proof}
  We will prove the proposition by induction on $n$, the initial case,
  $n=1$, is just \ref{propnon0}. Suppose that the proposition's
  assertion is true up to $n$, then  it is true for $n+1$, with
  effect:
  
  Ignoring $x_{n+1}$ we obtain that $f$ is either $f\equiv
  \Sigma_jP_1^j\cdots P_n^j p^j$ or $f\equiv \Sigma_jz^j\cdot
  P_1^j\cdots P_n^j$. If an $x_{n+1}$ appear in a $p$ then we have
  that $p\equiv p_1x_{n+1}$ or $p\equiv p_1x_{m+1}p_2$ where $g(p_1)=
  g(p_2)= g(x_{n+1})$ by \ref{lemnon0}. Using the same substitution
  arguments we may assume, without loss of generality, that $p\equiv
  p_1x_{n+1} p_2$. If an $x_{n+1}$ appear in a $z$ then we have that
  $p\equiv z_1x_{n+1}$ or $p\equiv z_1\cdot p_1x_{n+1}$ where $g(p_1)=
  g(x_{n+1})$ and $\langle g(z_1),g(x_{n+1})\rangle=\mathbb{Z}_2^2$,
  by \ref{lemnon0}. Using the same substitution arguments we may
  assume, without loss of generality, that $p\equiv z_1\cdot p_1
  x_{n+1}$. The proposition now follows from induction in $m_{n+1}$,
  the degree of $f$ with respect to $x_{n+1}$.
\end{proof}

\section{Coup de Gr\^ace}

\begin{rem}\label{rem}
  Let $u$ be a monomial that depends only on two components, both of
  them non-zero. Then $u\equiv\pm wv$ where $w$ is a monomial that
  depends only on one component and $v$ is a monomial that depends
  only on the other component.
\end{rem}
\begin{proof}
  A simple proof by induction on the degree of the monomial.
\end{proof}

\begin{lem}\label{lemcompnon0}
  Let $f$ be a multihomogeneos polynomial, suppose that $f\in
  T_{\mathbb{Z}_2^2}(\mathbb{O})\setminus I$ of minimal degree and $F$
  is an infinite field. Then $f$ depends on a zero component
  variable.
\end{lem}
\begin{proof}
  Assume, by contradiction, that $f$ does not depend on the zero
  component variable. Let $h\in G\setminus(0)$ s.t. $\forall y\in G
  \setminus(0)$, $\deg_hf\geq\deg_yf$. Then $f\equiv \Sigma_jP_1^j
  \cdots P_n^j p^j$ or $f\equiv \Sigma_jz^j\cdot P_1^j\cdots P_n^j$,
  where $P_i^j= p_{i,1}^jx_i\cdots p_{i,m_i}^jx_i$, $g(p_{i,l}^j)=
  g(p^j) =h$, $\langle g(z^j),g(x_i) \rangle=\mathbb{Z}_2^2$,
  $g(x_i)=h$, $i=1,\ldots,m$, $\langle 
  g(z^j),h\rangle=\mathbb{Z}_2^2$ and $m=\Sigma_{i=1}^nm_i=\deg_hf$,
  by \ref{totalnon0}.
  
  Furthermore we have that $p_{i,l}^j \equiv \alpha_{i,l}^j z_{i,l}^j
  w_{i,l}^j$ and $z^j \equiv \alpha^{j'} z^{j'} w_l^{j'}$  where
  $\alpha_{i,l}^j$, $\alpha^{j'} \in F$ and all the $z_{i,l}^j$'s and
  $z^{j'}$'s are product of variables from the same non zero component
  that is not $h$ and the $w_{i,l}^j$'s are products of variables from
  the third non zero component. Therefore $deg_{g(z)}f>deg_hf$ which
  is a contradiction.
\end{proof}

\begin{prop}\label{sem-homegeneo}
  Let $f$ be a multihomogeneous polynomial, assume that $f\in
  T_{\mathbb{Z}_2^2}(\mathbb{O})$ and $F$ is an infinite field. Then
  $f\in I$.
\end{prop}
\begin{proof}
  Assume, by contradiction, that $f\notin I$ so we can assume without
  loss of generality that $f$ is of minimal degree.
  
  We have that $\deg_0f=1$ by \ref{comp0sotem1} and \ref{lemcompnon0},
  so $f\equiv\Sigma_iy_ix\cdot z_i$ where $g(x)=0\neq g(y_i)=g(z_i)$
  and the $y$'s, $z$'s are free from the zero component by
  \ref{prop0}. Let $h \in \mathbb{Z}_2^2 \setminus (0)$ such that
  $\forall \alpha \in \mathbb{Z}_2^2 \setminus (0)$, $\deg_hf \geq
  \deg_{\alpha}f$ and $w$ a variable s.t. $g(w)=1$, $\deg_wf>0$.
  Applying \ref{formanon0}, \eqref{com}, \eqref{ass},
  \eqref{cocomuta2}, \eqref{entdir}, \eqref{lokao2} and \eqref{antcom}
  on the $y$'s and $z$'s we obtain that $f$ is of the form:
  \begin{align*}
    f\equiv & wx \Sigma_{i=1}^{n-1} \Sigma_j p_{i,i+1}^j w \cdots
    p_{i,1}^j w^{n-i-1} + p_1x \Sigma_{i'=1}^n \Sigma_j w p_{i',i'}^j
    w \cdots p_{i',1}^j w^{n'-i} + \\
    & wxk_1w^{n-1} + p_1xk_2w^n\\
  \end{align*}
  where the $p$'s are from the $h$ component and the $k$'s are either
  zero or from the $h$ component.
  
  Applying the usual substitution argument and the counting argument
  from \ref{formanon0}  we obtain that $\Sigma_j p_{i,i+1}^j w \cdots
  p_{i,1}^j$, $p_1x \Sigma_{i'=1}^n \Sigma_j w p_{i',i'}^j w \cdots
  p_{i',1}^j$, $k_1$ and $p_1xk_2$ are identities therefore in
  $I$. We may assume, whiteout loss of generality, that $f \equiv
  \Sigma_jp_j x w p_{n-1,n-1}^j \cdots w p_{n-1,1}^j \equiv
  (\Sigma_jp_j w p_{n-1,n-1}^j \cdots w p_{n-1,1}^j) x$.  Which is a
  contradiction.
\end{proof}

\begin{theo}
  If $F$ is an infinite field, then $I=T_{\mathbb{Z}_2^2}(\mathbb{O})$.
\end{theo}

\begin{theo}
  Let $D$ be an infinite domain, and form the ``Cayley-Dickson''
  algebra over $D$, $\mathbb{O}$. Then $I=T_{\mathbb{Z}_2^2}
  (\mathbb{O})$.
\end{theo}
\begin{proof}
  A direct application of \ref{central}
\end{proof}

Remembering that two split composition algebra of the same dimension
are isomorphic we see that $M_2(F) \cong \mathbf{Q}(0,1)$. If one
pushes the $\mathbb{Z}_2$ grading over the isomorphism he gets that
the zero component is formed by the diagonal matrices, and the unitary
component by the anti-diagonal matrices. More generally, let
$M_n(F)_\alpha := $lin.span$ \{  e_{i,j} | j-i \equiv \alpha
\pmod{n} \}$. For that we have the following:

\begin{theo}
  Let $F$ be a field of characteristic zero. Then $T_{\mathbb{Z}_2}
  (M_n(F))$ is generated as a $T$-ideal by associativity and the
  following identities:
  \begin{align}
    xy-yx=0 \qquad& g(x)=g(y)=0;\label{M_n1}\\
    x_1xx_2-x_2xx_1=0 \qquad& g(x_2)=g(x_1)=-g(x).\label{M_n2}
  \end{align}
\end{theo}

This theorem was first proved by Di Vincenzo in \cite{vinc} for two by
two matrices, latter Vasilovsky extended the proof for matrices of any
order in \cite{vasi}. Afterwards Koshlukov and Azevedo proved the
theorem for two by two matrices over an infinite field of
characteristic grater then two in \cite{kosh1}. Finally in
\cite{kosh2} Brand{\~a}o, Koshlukov and Krasilnikov remarked that the
proof in \cite{kosh1} is still valid for an infinite integral domain,
that is:

\begin{theo}
  Let $D$ be an infinite domain. Then $T_{\mathbb{Z}_2} (M_2(D))$ is
  generated as an $T$-ideal by the identities \eqref{M_n1},
  \eqref{M_n2} and associativity.
\end{theo}

Which we now re-obtain:
\begin{proof}
  Let's write the identities that generate $I$, but restricted to
  $\mathbb{Z}_2$:
  \begin{align*}
  ab\cdot v=v\cdot ba, & \qquad
  g(v)= g(a)= g(b) =\bar{1}; \tag{5*}\\
  (ax\cdot b)v=v(ba\cdot x), & \qquad
  g(x)= 0,\ g(a)= g(b)= g(v)= \bar{1}; \tag{6*} \\
  v(ax\cdot b)=(ba\cdot x)v, & \qquad
  g(x)= 0,\ g(a)= g(b)= g(v)= \bar{1}; \tag{7*} \\
  vb\cdot a=v\cdot ab, & \qquad
  g(v)= \bar{1},\ g(a)= g(b) =0; \tag{9*} \\
  (x,y,z)=0, & \qquad
  \tag {13*} \\
  [x,y]=0, & \qquad
  g(x)=g(y)=0; \tag{14*} \\
\end{align*}

Equations (8), (10), (11) and (12) do not intersect the $\mathbb{Z}_2$
realm. (13*) is associativity, (14*) and (5*) are respectively
\eqref{M_n1} and \eqref{M_n2}. Substituting $a$ for $ax$ in (5*) and
using associativity we obtain (6*), simultaneously substituting $a$
for $v$, $b$ for $ax$ and $v$ for $b$ in (5*) and using associativity
we obtain (7*), finally multiplying (14*) by $v$ and using
associativity we obtain (9*).
\end{proof}

\section{Acknowledgments}  
I would like to thank FAPESP for it's financial support (process
number: 2009/51920-7) and most of all I'd like to thank Professor
Ivan P. Shestakov, for not only his great advice, which this work is
done under, but also the patience he had in advising me over the
years.

\bibliographystyle{amsalpha}
\bibliography{bibliografia2}

\end{document}